\documentclass[12pt, reqno]{amsart}
\usepackage{binary}

\title[Binary Curves]{Binary Curves of Small Fixed Genus and Gonality with Many Rational Points}
\author{Xander Faber}
\author{Jon Grantham}
\address{Institute for Defense Analyses \\
Center for Computing Sciences \\
17100 Science Drive \\
Bowie, MD} 

\begin{document}

\begin{abstract}
  We determine the maximum number of rational points on a curve over $\FF_2$
  with fixed gonality and small genus. This document differs from the published
  edition in \textit{Journal of Algebra} in one significant way: it contains
  Appendix~B on quadratic forms over finite fields and their associated
  orthogonal groups.
\end{abstract}

\maketitle


\section{Introduction}

Let $C_{/\FF_q}$ be a smooth complete algebraic curve of genus $g$ over the
finite field with $q$ elements.  The Weil bound gives an upper limit on the
number of rational points on this curve in terms of $q$ and $g$:
\[
  \#C(\FF_q) \leq q + 1 + 2g \sqrt{q}.
\]
It is natural to ask how close we can get to this bound as we vary the curve
$C$. Serre, Oesterl\'e, Drinfeld-Vl\v{a}du\c{t}, and others gave improvements
to this bound in a variety of settings and constructed examples of curves that
in some cases meet these improved bounds. For a brief survey of this work, see
\cite{voight}; a more thorough treatment is given in
\cite{Niederreiter_Xing_book}. From a computational point of view, van der Geer
and van der Vlugt \cite{vandersquared} built the first comprehensive table of
maximal values for small genus and field size. Their table evolved into
\url{manypoints.org}.

We restrict our attention almost entirely to the binary field $\FF_2$ in the
present work. Write $N_2(g)$ for the maximum number of rational points on a
curve of genus $g$ over $\FF_2$. The values in Table~\ref{tab:manypoints} come
from \url{manypoints.org}.
\begin{table}[ht]
\begin{tabular}{c|c}
  $g$ & $N_2(g)$ \\
  \hline
  0 & 3 \\
  1 & 5 \\
  2 & 6 \\
  3 & 7 \\
  4 & 8 \\
  5 & 9 \\
\end{tabular}
\caption{Maximum number of rational points on binary curves with fixed genus}
\label{tab:manypoints}
\end{table}

The \textbf{gonality} of a curve $C$ over a field~$k$ is the minimum degree of a
$k$-morphism $C \to \PP^1$. As special cases, one often refers to curves with
genus $g \ge 2$ and gonality~2 as \textbf{hyperelliptic}, and curves with
gonality~3 as \textbf{trigonal}. We caution the reader that some authors use the
word ``gonality'' to mean the gonality of $\bar C = C \times_{\Spec k} \Spec
k^{\mathrm{sep}}$, and some authors use the word ``trigonal'' to mean ``admits a morphism
to $\PP^1$ of degree~3'', without the requirement that no morphism of degree
less than $3$ exists. We will also say that ``$C$ has a $g^1_d$'' if it has a
$\Gal(k^{\mathrm{sep}} / k)$-invariant divisor $D$ of degree-$d$ such that $\dim
|D| \geq 1$; if true, then the gonality of $C$ is at most~$d$.

Van der Geer \cite{coffac} asked,
\begin{quote}
\textit{What is the maximum number of rational points on a curve of genus $g$ and
gonality $\gamma$ defined over $\FF_q$?}
\end{quote}
We found no evidence in the literature that anyone has tried to answer this
question, so we are embarking on a program of filtering
Table~\ref{tab:manypoints} according to the gonality invariant. To that end,
define $N_2(g,\gamma)$ to be the supremum of the number of rational points on a
smooth proper connected curve over $\FF_2$ with genus~$g$ and
gonality~$\gamma$. We use the supremum in order to be able to make statements
like $N_2(1,3) = -\infty$; this means there is no genus-1 curve with
gonality~3. With regard to non-existence results, we found some valuable insight
in \cite{Castryck_Tuitman} on which gonalities can occur for a curve of small
genus over finite fields \textit{with odd characteristic}.

The remainder of the paper will be dedicated to explaining the entries that
appear in Table~\ref{tab:manygonality}, where we have limited our attention to
curves of genus at most~5. While this is somewhat arbitrary, it reflects the
limitation of our understanding of the geometry of curves of genus $g \geq
6$. But even for genera that we are explicitly considering, certain entries like
$N_2(1,3)$ do not appear: this is because of the general fact that a curve of
genus~$g$ over a finite field has gonality at most $g+1$
(Proposition~\ref{prop:gengon}). We omit entries beyond this bound from
Table~\ref{tab:manygonality}. Beyond our systematic approach to genus $g \le 5$,
some additional values of $N_2(g,\gamma)$ can be obtained from a variety ad hoc
methods; we give examples in Appendix~\ref{sec:misc}.

\begin{table}[th]
  \begin{tabular}{c|c|c|c}
  $g$ & $\gamma$ & $N_2(g,\gamma)$ & Reference \\
  \hline \hline
  0 & 1 & 3 & $\PP^1$ \\
  \hline
  1 & 2 & 5 & Equation~\eqref{eq:1-2} \\
  \hline
  2 & 2 & 6 & Theorem~\ref{thm:hyperelliptic} \\
  \hline
  3 & 2 & 6 & Theorem~\ref{thm:hyperelliptic} \\
   & 3 & 7 & Theorem~\ref{thm:3-3} \\
   & 4 & 0 & Theorem~\ref{thm:3-4} \\
  \hline
  4 & 2 & 6 & Theorem~\ref{thm:hyperelliptic} \\
   & 3 & 8 & Theorem~\ref{thm:4-3} \\
   & 4 & 5 & Theorem~\ref{thm:4-4} \\
   & 5 & 0 & Theorem~\ref{thm:4-5} \\
  \hline
  5 & 2 & 6 & Theorem~\ref{thm:hyperelliptic} \\
   & 3 & 8 & Theorem~\ref{thm:5-3} \\
   & 4 & 9 & Theorem~\ref{thm:5-4}  \\
   & 5 & 3 & Theorem~\ref{thm:5-5} \\
   & 6 & $-\infty$ & Theorem~\ref{thm:5-6} \\  
  \end{tabular}
  \caption{Supremum of the number of rational points on a binary curve with
    fixed genus and gonality.}
  \label{tab:manygonality}
\end{table}

Given a curve $C_{/\FF_2}$ with gonality $\gamma$, and a morphism $\pi : C \to
\PP^1$ of degree~$\gamma$, we observe that every rational point of $C$ must lie
over a rational point of $\PP^1$, and that over each such point there are at
most $\gamma$ rational points of $C$:
\begin{equation}
\label{eq:gamma_bound}
\#C(\FF_2) \leq \#\PP^1(\FF_2) \cdot \gamma = 3\gamma.
\end{equation}
This ``gonality-point inequality'' is substantially weaker than the truth when
$g$ is small; for example, $N_2(5,4) = 9$. It is also much weaker than the Weil
bound. If $C$ were an example of equality in \eqref{eq:gamma_bound}, we would
have
\[
3\gamma = \#C(\FF_2) \leq 3 + 2g\sqrt{2} \quad \Rightarrow \quad g \geq \frac{3(\gamma-1)}{2\sqrt{2}}.
\]
We conjecture that small genus is the only obstruction to achieving equality:

\begin{conjecture}
  Fix $\gamma \geq 2$. For $g$ sufficiently large, $N_2(g,\gamma) = 3\gamma$. 
\end{conjecture}

We give a simple construction that proves this conjecture in the case $\gamma =
2$ (Theorem~\ref{thm:hyperelliptic}). The case $\gamma = 3$ should require a bit
more ingenuity since a plane curve over $\FF_2$ has at most 7 rational
points.\footnote{Shortly before this paper was accepted for publication, we
  learned that Vermeulen has addressed the case $\gamma = 3$ using curves on
  toric surfaces \cite{vermeulen2021curves}.}

A curve $C$ over a field $k$ has gonality at most~$d$ if it admits a morphism $f
: C \to \PP^1$ of degree~$d$. If we write $D = f^{-1}(\infty)$, then $D$ is a
Galois-invariant divisor and the complete linear system $|D|$ has degree~$d$ and
dimension at least~1. In this way, we see that computing the gonality of $C$ can
be reduced to looking for Galois-invariant effective divisors that move in
families. A canonical divisor $K$ on $C$ is one candidate, and Galois orbits of
points in $C(k^{\mathrm{sep}})$ provide others. In Section~\ref{sec:gonality}, we assemble
a number of results (most of which are known) that will allow us to identify the
gonality of curves of small genus.

We treat curves of genus at most~3 in Sections~\ref{sec:at_most_2} and
\ref{sec:genus3}. The arithmetic of a non-hyperelliptic curve of genus $g \geq
4$ is intimately tied to the quadric hypersurfaces in $\PP^{g-1}$ on which it
lives. It is fruitful to understand these hypersurfaces via the classification
of quadratic forms. As we could not find a single self-contained reference that
gives all of the results we need in characteristic~2, we assemble the necessary
facts in Appendix~\ref{app:quadratic_forms}. Curves of genus~4 and~5 are studied
in Sections~\ref{sec:genus4} and~\ref{sec:genus5}.

Our work on curves of genus at most~4 lives squarely in the realm of pure
mathematics: we provide an upper bound for the number of rational points on a
curve with specified $g$ and $\gamma$ and then exhibit a curve that achieves
this bound. This also works for genus-5 curves when the gonality is at
most~4. However, we were unable to maintain this pattern for $g = 5$ and $\gamma
\geq 5$. Instead, we provide a certain amount of theoretical scaffolding, and
then complete the edifice with an exhaustive computation in Sage
\cite{sage_8.7}. A curious non-existence result comes out of these computations:

\begin{theorem}
There is no curve of genus~5 and gonality~6 over $\FF_2$. 
\end{theorem}

General theory of gonality on curves over finite fields shows that a curve of
genus~$g$ has gonality at most $g+1$ (see \S\ref{sec:gonality}), but there is no
guarantee that any curve of gonality $g+1$ exists. In fact, Weil's (lower)
inequality in tandem with Corollary~\ref{cor:gonality5} proves that no curve of
genus~5 and gonality~6 over $\FF_q$ exists as soon as $q > 4$. The above theorem
deals with one of the cases the Weil bound misses. We note that the question of
whether the above theorem holds over $\FF_3$ appears in
\cite[Rem.~3.14]{Castryck_Tuitman}.

We draw attention to two results in \S\ref{sec:genus5_gen} that we expect to be
of value to a wider audience. Over an algebraically closed field, it is known
that a canonically embedded curve of genus~5 is the complete intersection of
quadric hypersurfaces if and only if it is non-trigonal. With the help of Galois
cohomology, we extend this result to an arbitrary perfect field
(Theorem~\ref{thm:genus5_nontrig}). In particular, this shows that a curve of
genus~5 over a perfect field is trigonal if and only if it is geometrically
trigonal (Corollary~\ref{cor:perfect_trigonal}). We believe this to be a new
result for fields of cohomological dimension larger than~1.

While our results are focused on binary curves, many of the results in this
article apply equally well to curves over more general fields. We explicitly
state when a result applies more broadly. Unless otherwise specified, in this
article a \textbf{curve} $C_{/k}$ is a smooth proper geometrically irreducible
scheme of dimension~1 over a field $k$. A divisor on $C$ will always be defined
over the ground field $k$. For a divisor $D$ on $C_{/k}$, we write $L(D)$ for
the $k$-vector space of rational functions $f$ whose divisor satisfies $\div(f)
\geq -D$. The set of effective divisors that are linearly equivalent to $D$ is
denoted $|D|$, and we recall that when it is nonempty, it admits the structure
of a projective space with $\dim |D| = \dim L(D) - 1$.


\section{Generalities on gonalities}
\label{sec:gonality}

Throughout this section, we work over a fixed finite field $\FF_q$.

The following results are well known, at least in some form --- see the appendix
of \cite{Poonen_gonality}, for example. We collect them here for ease of use
later. We begin with general bounds for the gonality in terms of the genus, thus
limiting the number of entries that must appear in Table~\ref{tab:manygonality}.

\begin{proposition}
  \label{prop:gengon}
  Let $C_{/\FF_q}$ be a curve of genus~$g$ and gonality $\gamma$.
  \begin{enumerate}
  \item If $g = 0$, then $\gamma = 1$.
  \item If $g \in \{1,2\}$, then $\gamma = 2$.
  \item If $g \ge 2$ and $C(\FF_q) \ne \varnothing$, then $\gamma \leq g$.
  \item In general, the gonality satisfies $\gamma \leq g+1$.     
  \end{enumerate}
\end{proposition}

\begin{proof}
  Every curve of genus zero over a finite field has a rational point. (Recall
  that the anti-canonical linear system $|-K|$ has degree~2 and dimension~2, so
  a genus-zero curve is cut out by an irreducible quadratic form in 3
  variables. Every such form has a nonzero solution over a finite field.) Linear
  projection through the rational point gives an isomorphism between the curve
  and $\PP^1$. That is, $\gamma = 1$ when $g = 0$.

  Every curve of genus~1 over a finite field also has a rational point. Indeed,
  the Weil bound takes the form
  \[
    |\#C(\FF_q) - (q+1)| \leq 2\sqrt{q},
    \]
  and $q+1 > 2\sqrt{q}$ for all $q > 1$. Choosing $P \in C(\FF_q)$, Riemann-Roch
  implies that the linear system $|2P|$ has dimension~1. Hence, $\gamma = 2$.

  For a curve of genus~2, Riemann-Roch shows that the canonical linear system
  $|K|$ has degree~2 and dimension~1. So $\gamma = 2$.
  
  Suppose that $g \geq 2$, and let $P \in C(\FF_q)$. The linear system $|K -
  (g-2)P|$ has degree~$g$. Since $|(g-2)P|$ is nonempty, Riemann-Roch gives the
  inequality
  \[
  \dim |K - (g-2)P| = \dim |(g-2)P| + 1 \geq 1.
  \]
  Hence, $\gamma \leq g$. 
  
  Finally, we show that $\gamma \leq g+1$ in general for a curve $C$ over a
  finite field. This statement appears as
  \cite[Cor.~4.2.18]{TVN_alg_geom_codes}, though the proof is
  incomplete.\footnote{The authors of \cite{TVN_alg_geom_codes} assert that
    every curve over a finite field $\FF_q$ has a point defined over an
    extension of degree~$(g~+~1)$. While this is true, it does not follow from the
    Weil bound when $q = 2$ and $2 \le g \le 6$. We thank Felip\'e Voloch for
    pointing us toward Schmidt's proof.} A
  result of F.K. Schmidt from 1931 \cite[Cor.~3.1.12]{TVN_alg_geom_codes}
  asserts that there is a (Galois-invariant) divisor $D$ on $C$ of degree~1. The
  linear system $|(g+1)D|$ has degree $g+1$, and Riemann-Roch shows it has
  dimension
  \[
    \dim |(g+1)D| = \dim |K-(g+1)D| + 2 \geq 1. \qedhere
    \]
\end{proof}

We now give some refinements of this result which will help us to determine the
gonality of certain examples we come across.



\begin{proposition}
  \label{prop:g1g}
  Let $C_{/\FF_q}$ be a curve of genus $g \geq 2$. Then $C$ admits a $g^1_g$ if
  and only if there is an effective divisor $D$ on $C$ of degree $g - 2$.
\end{proposition}

\begin{proof}
  If $C$ has a $g^1_g$, then there is an effective divisor $D_0$ of degree $g$
  such that $\dim |D_0| \ge 1$. By Riemann-Roch, we see that
  \[
  \dim |K - D_0| = \dim |D_0| - 1 \ge 0.
  \]
  Let $D$ be an effective divisor in the linear system $|K - D_0|$. Then $D$ has
  degree $g-2$, as desired. Conversely, if $D$ is an effective divisor on $C$ of
  degree $g-2$, then
  \[
    \dim |K - D| \ge \dim |K| - \deg(D) = 1.
    \]
  That is, $|K - D|$ contains a $g^1_g$. 
\end{proof}

\begin{corollary}
  \label{cor:genus3_gonality4}
  Let $C_{/\FF_q}$ be a non-hyperelliptic curve of genus~3. If $C(\FF_q)$ is nonempty,
  then $C$ has gonality~3; otherwise, $C$ has gonality~4.
\end{corollary}

\begin{proof}
  Proposition~\ref{prop:g1g} implies that $C$ has a $g^1_3$ if and only if it
  has an effective divisor of degree~1 --- i.e., a rational point. If $C$ does
  not have gonality~3, it must have gonality~4 by Proposition~\ref{prop:gengon}
\end{proof}

\begin{corollary}
  \label{cor:genus4_gonality4}
  Let $C_{/\FF_q}$ be a curve of genus~4 that is not hyperelliptic or
  trigonal. Then $C$ has gonality $4$ if $C(\FF_{q^2}) \ne \varnothing$, and
  otherwise it has gonality~5.
\end{corollary}

\begin{proof}
  We claim that $C$ has a $g^1_4$ if and only if $C(\FF_{q^2}) \ne \varnothing$.
  Proposition~\ref{prop:g1g} asserts that $C$ has a $g^1_4$ if and only if it
  admits an effective divisor $D$ of degree~2. If $D$ is supported at a single
  point $P$, then $P$ is rational. Otherwise, the support of $D$ is a single
  Galois orbit of quadratic points. This proves the claim.

  Since $C$ is not hyperelliptic or trigonal, we see that it has gonality~4 if
  $C(\FF_{q^2}) \ne \varnothing$, and otherwise it has gonality~5 by
  Proposition~\ref{prop:gengon}.
\end{proof}

\begin{corollary}
  \label{cor:gonality5}
  Let $C$ be a curve of genus~5 over $\FF_q$ with gonality at least 5. Then $C$
  has gonality~5 if $C(\FF_{q^3}) \ne \varnothing$, and otherwise it has gonality~6.
\end{corollary}

\begin{proof}
  As in the previous proof, it suffices to show that $C$ has a $g^1_5$ if and
  only if $C(\FF_{q^3}) \ne \varnothing$. Proposition~\ref{prop:g1g} asserts
  that $C$ has a $g^1_5$ if and only if there is an effective divisor $D$ on $C$
  of degree $3$. If $D$ is supported at a single point $P$, then that point must
  be rational. If $D$ is supported at two distinct points $P$, $Q$, then without
  loss we have $D = 2P + Q$, and the Galois invariance of $D$ implies that both
  $P$ and $Q$ are rational. Finally, suppose that $D = P + Q + R$ for three
  distinct points $P,Q,R$. If none of these points is rational, then they
  constitute a Galois orbit of length 3. In all cases, we find that
  $C(\FF_{q^3}) \ne \varnothing$.
\end{proof}


\section{Curves of genus at most~2}
\label{sec:at_most_2}

Proposition~\ref{prop:gengon} allows us to address curves of small genus rather
quickly.  A curve $C_{/\FF_2}$ of genus~0 is isomorphic to $\PP^1_{\FF_2}$, and
hence, $\#C(\FF_2) = 3$. That is, $N_2(0,1) = 3$.

A curve $C_{/\FF_2}$ of genus~1 has gonality~2. The Weil bound shows that
$\#C(\FF_2) \leq 2 + 1 + 2\sqrt{2} < 6$. Hence $N_2(1,2) \leq 5$. The following
example of an elliptic curve shows that this bound is sharp:
\begin{equation}
  \label{eq:1-2}
E_{/\FF_2} : y^2 + y = x^3 + x.
\end{equation}

A curve $C_{/\FF_2}$ of genus~2 is hyperelliptic. The following result, applied
when $g=2$, tells us that $N_2(2,2) = 6$.

\begin{theorem}
  \label{thm:hyperelliptic}
Fix an integer $g \ge 2$, and let $\delta \in \{0,1\}$ satisfy $\delta \equiv g
\pmod 2$.  The curve $C_{/\FF_2}$ with affine plane equation
  \[
  y^2 + \left( x^{g+1} + x^g + 1\right) y = \left[x(x+1)\right]^{g -
    \delta}
  \]
is hyperelliptic of genus~$g$ and has~6 rational points. In particular,
$N_2(g,2) = 6$.
\end{theorem}

\begin{proof}  
  A model for $C$ near infinity is given by setting $y = z / w^{g+1}$ and $x =
  1/w$:
  \[
    z^2 + \left[w^{g+1} + w + 1\right]z = w^{2+2\delta} (1+w)^{g-\delta}.
    \]
  The description of hyperelliptic curves in
  \cite[Prop.~7.4.24]{Qing_Liu_Algebraic_Geometry} shows that $C$ is smooth of
  genus~$g$, and one sees immediately that $C$ has 4 affine rational points and
  2 rational points at infinity (with $w = 0$).

  This example shows that $N_2(g,2) \geq 6$. The opposite inequality follows
  from the gonality-point inequality \eqref{eq:gamma_bound}.
\end{proof}


\section{Curves of genus~3}
\label{sec:genus3}

Proposition~\ref{prop:gengon} shows that a curve $C_{/\FF_2}$ of genus~3 has
gonality at most~4. The case of a hyperelliptic curve is immediately dispensed
with by Theorem~\ref{thm:hyperelliptic}. Every other curve of genus~3 is
canonically embedded as a quartic curve in $\PP^2$; conversely, the adjunction
formula shows that any smooth quartic curve in $\PP^2$ has genus~3.  (In
general, the canonical class is defined over the ground field, so the canonical
embedding will be as well.)

\begin{theorem}
  \label{thm:3-3}
  $N_2(3,3) = 7$.
\end{theorem}

\begin{proof}
Dickson observed that the plane curve
\[
C_{/\FF_2} : x^3 y + x^2 y^2 + xz^3 + x^2z^2 + y^3z + yz^3 = 0
\]
is the unique nonsingular quartic (up to isomorphism) that passes through all 7
rational points of the projective plane \cite[\S5]{Dickson_Quartics}. As it is
smooth, it has genus~3. Corollary~\ref{cor:genus3_gonality4} shows that $C$
has gonality~3. Thus, $N_2(3,3) \geq 7$.

For the reverse inequality, let $C$ be any curve of genus~3 and gonality~3,
which we may identify with its image in $\PP^2$ under a canonical embedding. The
projective plane over $\FF_2$ has 7 rational points, so we obtain $N_2(3,3) \leq
7$.
\end{proof}

\begin{theorem}
  \label{thm:3-4} $N_2(3,4) = 0$.
\end{theorem}

\begin{proof}
  In \cite[\S4]{Dickson_Quartics}, Dickson observed that the plane quartic curve
  \[
  C_{/\FF_2} : x^4 + y^4 + z^4 + x^2y^2 + x^2z^2 + y^2z^2 + x^2yz + xy^2z +
  xyz^2 = 0
  \]
  has no $\FF_2$-rational point. One verifies easily that it is nonsingular,
  which means $C$ has genus~3. Corollary~\ref{cor:genus3_gonality4} shows
  that $C$ has gonality~4, so $N_2(3,4) \geq 0$. The same proposition shows that
  $N_2(3,4) \leq 0$.
\end{proof}





\section{Curves of genus~4}
\label{sec:genus4}

For a curve of genus~4, Proposition~\ref{prop:gengon} shows that the gonality is
at most 5. Our approach will serve as a paradigm for the more difficult case of
genus-5 curves.

Hyperelliptic curves are handled by Theorem~\ref{thm:hyperelliptic}, where we
learned that $N_2(4,2) = 6$. The key geometric facts about non-hyperelliptic
curves of genus~4 are contained in the following result. We sketch a proof as we
could not easily reconstruct one from the available literature. (The result is
stated, for example, in \cite[\S4.4]{DLR_Cubic_Hypersurfaces}.)

\begin{lemma}
  \label{lem:gen4}
  Let $C$ be a genus-4 curve over a finite field $\FF_q$. If $C$
  is not hyperelliptic, then the canonical linear system $|K|$ embeds $C$ into
  $\PP^3_{\FF_q} = \Proj \FF_q[x,y,z,w]$ as the intersection of a unique quadric
  surface $S$ and a cubic surface. Exactly one of the following is true, up to
  automorphism of $\PP^3$:
  \begin{enumerate}
  \item $C$ admits a $g^1_3$, say $|D|$, such that $D \not\sim K - D$. Then $S =
    \{xy + zw = 0\}$, and $S$ is isomorphic to $\PP^1 \times \PP^1$ under the
    Segre embedding. The two linear systems $|D|$ and $|K - D|$ correspond to
    the two rulings on $\PP^1 \times \PP^1$, and $C$ corresponds to a smooth
    curve of bidegree~$(3,3)$.
    
  \item $C$ admits a $g^1_3$, say $|D|$, such that $D \sim K - D$. Then $S =
    \{xy + z^2 = 0\}$, a quadric cone, and the family of lines through the
    singularity cuts out the unique $g^1_3$ on $C$.

  \item $C$ admits no $g^1_3$. In this case, $S = \{xy + N(z,w) = 0\}$, where
    $N(z,w)$ is the norm form for $\FF_{q^2}/\FF_q$, as in \eqref{eq:norm_form}.
  \end{enumerate}
\end{lemma}

\begin{proof}
  As $C$ is not hyperelliptic, the canonical linear system is very ample of
  degree~6 and dimension~3. The image of the corresponding morphism $C
  \hookrightarrow \PP^3$ is the intersection of a \textit{unique} quadric
  surface $S = \{Q=0\}$ and a cubic surface $\{F = 0\}$
  \cite[IV.5.2.2]{Hartshorne_Bible}. The classification of quadratic forms in at
  most 4 variables (Theorem~\ref{thm:classification}) shows that, up to linear
  change of coordinates and rescaling, $Q$ may be taken to be among the following
  \begin{itemize}
  \item (4 variables) $xy + zw$ or $xy + N(z,w)$;
  \item (3 variables) $xy + z^2$;
  \item (2 variables) $xy$ or $N(x,y)$;
  \item (1 variable) $x^2$.
  \end{itemize}
  The forms in 1 or 2 variables are geometrically reducible; as $C$ is
  not contained in a hyperplane of $\PP^3$, none of these cases can occur. The
  three remaining surfaces are pairwise non-isomorphic over~$\FF_q$.

  Suppose that $C$ admits a $g^1_3$, say $|D|$. By Riemann-Roch, $|K-D|$ is
  also a $g^1_3$. As we have assumed that $C$ is not hyperelliptic, neither of
  these linear systems has a basepoint. We need to consider separately the cases
  $D \sim K - D$ and $D \not\sim K - D$.

  Assume that $D$ is a $g^1_3$ such that $D \sim K - D$. Let $f \in L(D)$ be a
  nonconstant rational function with $\div(f) \geq - D$. The divisor of poles of
  $f$ is precisely $D$, else $C$ would be rational or hyperelliptic. Then the
  Riemann-Roch space $L(2D) \cong L(K)$ has dimension $4$, and it contains the
  linearly independent functions $1$, $f$, and $f^2$. Let $h \in L(2D)$ complete
  this to a basis. Define a morphism by
  \begin{eqnarray*}
    C \smallsetminus \supp(D) &\to& \PP^3 = \Proj \FF_q[x,y,z,w] \\
    P &\mapsto & (-1, f^2(P), f(P), h(P)).
  \end{eqnarray*}
  Note that the image lies on the surface $xy + z^2 = 0$. As $C$ is smooth, this
  morphism extends over all of $C$, and it is clear by its definition that it
  yields the canonical embedding of $C$.

  Now assume that $D$ is a $g^1_3$ with $D \not\sim K - D$. Consider the
  morphism $C \to \PP^1 \times \PP^1$ induced by $|D|$ and $|K-D|$. Applying the
  Segre embedding gives a composition
  \[
  C \to \PP^1 \times \PP^1 \hookrightarrow \PP^3,
  \]
  and if we choose the coordinates correctly, the image lies on the quadric
  surface $\{xy + zw = 0\}$. Intersecting a hyperplane with this surface gives a
  bidegree-$(1,1)$ divisor on $\PP^1 \times \PP^1$, and pulling it back to $C$
  gives $D + K-D = K$. That is, this composition corresponds to the canonical
  linear system. In order to show that the map $C \to \PP^3$ is a canonical
  embedding, it remains to check that it corresponds to the complete linear
  system $|K|$, or equivalently, that the image of $C$ does not lie on a
  hyperplane in $\PP^3$. Suppose otherwise. Since a hyperplane pulls back to a
  $(1,1)$-divisor on $\PP^1\times \PP^1$, we conclude that the image of $C$ in
  $\PP^1 \times \PP^1$ is a $(1,1)$-curve, say $C_0$. Write $\psi : C \to C_0$
  for the induced morphism. By the adjunction formula, the curve $C_0$ has
  genus~0. Write $\pi_1, \pi_2$ for the morphisms $C_0 \to \PP^1$ induced by the
  component maps on $\PP^1 \times \PP^1$. Since $C_0$ is a rational curve,
  $\pi_1^*[\infty] \sim \pi_2^* [\infty]$, and we find that
  \[
  D \sim \psi^* \pi_1^*[\infty] \sim \psi^* \pi_2^*[\infty] \sim K - D,
  \]
  which is a contradiction. Thus $C$ does not lie on a hyperplane in $\PP^3$.

  Finally, we are left with the case where $C$ has no $g^1_3$. If, under the
  canonical embedding, $C$ lies on $xy + zw = 0$, then it would have a $g^1_3$
  coming from either of the rulings on this surface. If instead, $C$ lies on $xy
  + z^2 = 0$, then the lines through the singularity of this quadric cone give
  rise to a pencil of degree-3 divisors on $C$, and hence a $g^1_3$. So $C$ must
  lie on the remaining surface $\{xy + N(z,w) = 0\}$.
\end{proof}

We now return to the case of curves over $\FF_2$. Note that the norm form for
$\FF_4 / \FF_2$ is given by $N(z,w) = z^2 + zw + w^2$. Consequently, any canonically
embedded curve must lie on one of the quadrics $\{xy + zw = 0\}$, $\{xy + z^2 =
0\}$, or $\{xy + z^2 + zw + w^2 = 0\}$. 

\begin{theorem}
  \label{thm:4-3}
$N_2(4,3) = 8$
\end{theorem}

\begin{proof}
  Consider the genus-4 curve $C \subset \PP^3_{\FF_2}$ described by the
  following equations:
  \begin{eqnarray*}
    xy + zw &=& 0\\
    xy^2 + y^3 + x^2z + y^2z + xz^2 + x^2w + y^2w + xw^2 &=& 0.
  \end{eqnarray*}
  One verifies that $C$ passes through 8 of the nine points on the quadric
  surface $xy + zw = 0$. Lemma~\ref{lem:gen4} shows that $C$ has gonality~3, 
  so $N_2(4,3) \geq 8$. Serre has shown that $N_2(4) \le 8$
  \cite{Serre_points_on_curves_1983}. In particular, $N_2(4,3) \leq 8$.
\end{proof}

\begin{theorem}
  \label{thm:4-4}
$N_2(4,4) = 5$.
\end{theorem}

\begin{proof}
  We begin with the genus-4 curve $C_{/\FF_2}$ in $\PP^3$ given by the equations
  \begin{eqnarray*}
    xy + z^2 + zw + w^2 &=& 0  \\
    xy^2 + x^2z + y^2z + yz^2 + x^2w + z^2w &=& 0.
  \end{eqnarray*}
  A quick calculation shows that the cubic form vanishes at all five rational
  points of the quadratic form. Thus $\#C(\FF_2) = 5$.  Lemma~\ref{lem:gen4}
  shows that it does not admit a morphism to $\PP^1$ of degree~3. Thus, its
  gonality must be at least~4.  On the other hand, Proposition~\ref{prop:gengon}
  shows that, since it has a rational point, its gonality is at most~4. This
  proves $N_2(4,4) \geq 5$.

  For the reverse inequality, observe that every non-hyperelliptic genus-4 curve
  can be embedded in $\PP^3$ using the canonical linear system. By
  Lemma~\ref{lem:gen4}, we may assume it lies on the quadric surface $xy + z^2 +
  wz + w^2 = 0$ (else it would admit a degree-3 map to the projective
  line). This surface has only 5 rational points, so $N_2(4,4) \leq 5$.
\end{proof}

\begin{theorem}
 \label{thm:4-5}
 $N_2(4,5) = 0$.
\end{theorem}

\begin{proof}
  Consider the curve $C_{/\FF_2}$ in $\PP^3$ cut out by the equations
  \begin{eqnarray*}
    xy + z^2 + zw + w^2 &=& 0  \\
    x^3 + y^3 + z^3 + y^2w + xzw &=& 0.
  \end{eqnarray*}
  They define a curve of genus 4, and by Lemma~\ref{lem:gen4}, $C$ has gonality
  at least~4.
  
  Direct search using the above equations shows that $\#C(\FF_4) =
  0$. Corollary~\ref{cor:genus4_gonality4} shows that $C$ has no $g^1_4$, and
  hence gonality~5. That is, $N_2(4,5) \geq 0$.  

  For the reverse inequality, we recall that any curve of genus~4 with a
  rational point has gonality at most 4. It follows that a curve of gonality~5
  has no rational point, and $N_2(4,5) \leq 0$. 
\end{proof}

\section{Curves of genus~5}
\label{sec:genus5}

According to Proposition~\ref{prop:gengon}, a curve of genus 5 can have gonality
$\gamma$ satisfying $2 \leq \gamma \leq 6$. The hyperelliptic case is taken care
of by Theorem~\ref{thm:hyperelliptic}: $N_2(5,2) = 6$. A trigonal curve of genus
5 can be expressed as a singular plane quintic; we use this fact to compute
$N_2(5,3)$ in the next subsection.  Non-trigonal curves of genus~5 over an
algebraically closed field are well understood, and with some additional effort
we are able to extend this description to arbitrary fields. (This issue is
touched upon in \cite[p.36]{Castryck_Tuitman}, but they abdicate responsibility
in the case of characteristic~2.) We give a general discussion in
\S\ref{sec:genus5_gen}, and we use this theory to execute the calculation of
$N_2(5,4)$, $N_2(5,5)$, and $N_2(5,6)$ in \S\ref{sec:5-4} and
\S\ref{sec:genus5_calc}.


\subsection{Trigonal curves of genus 5}

Here is a useful fact about trigonal curves; see
\cite[\S2.2]{Kudo-Harashita_trigonal}. Note that this is incorrectly stated in
Exercise~IV.5.5 of \cite{Hartshorne_Bible}: the cuspidal case can actually
occur.

\begin{lemma}
  \label{lem:triggenus5}
  Every trigonal curve $C$ of genus~5 over a field $k$ is birational to a plane
  quintic $C' \subset \PP^2_k$ with a unique $k$-rational singularity of
  multiplicity~2. After moving the singularity to $(0:0:1) \in \PP^2(k)$, a
  homogeneous equation for $C'$ can be given by an irreducible polynomial
  \[
  f(x,y,z) = f_2(x,y)z^3 + g(x,y,z),
  \]
  where $f_2$ is a quadratic form, $g$ is a quintic form that vanishes
  to order~3 at $(0:0:1)$, and one of the following holds:
  \begin{itemize}
    \item Cusp: $f_2(x,y) = x^2$ and $g$ has nonzero coefficient on the $y^3z^2$-term;
    \item Split node: $f_2(x,y) = xy$; 
    \item Nonsplit node: $f_2(x,y)$ is irreducible over $k$.
  \end{itemize}
  Conversely, the normalization of any quintic curve in $\PP^2$ satisfying the
  above conditions is a trigonal curve of genus~5. 
\end{lemma}

\begin{proof}
  This is a consequence of the theorems of Riemann-Roch and
  Clifford. See \cite[Prop.~2.2]{Kudo-Harashita_trigonal} for the case of
  odd characteristic; the even characteristic case is virtually identical. 
\end{proof}

\begin{theorem}
  \label{thm:5-3}
$N_2(5,3) = 8$.
\end{theorem}

\begin{proof}
  Consider the singular curve $C' \subset \PP^2 = \Proj \FF_2[x,y,z]$ given by 
  \[
    xyz^3 + x^3z^2 + y^3z^2 + x^4z + xy^3z + y^4z + x^4 y + x^2 y^3 = 0.
  \]
One verifies that its only singularity is the node at $(0:0:1)$ and that it
passes through all $7$ rational points of the plane. After blowing up the node,
we obtain a curve $C$ of genus~5 with 8 rational points. By the converse part of
Lemma~\ref{lem:triggenus5}, $C$ is trigonal, so that $N_2(5,3) \geq 8$.

For the reverse inequality, let $C_{/\FF_2}$ be a trigonal curve of genus~5, and
let $C'$ be a quintic in $\PP^2$ that is birational to $C$. Without loss, we may
move the unique singularity of $C'$ to (0:0:1. Let $\pi : C \to C'$ be the
normalization morphism. Since $\#\PP^2(\FF_2) = 7$, it follows that the number
of rational points on $C$ is at most $6$ plus the number of rational points on
$\pi^{-1}(0:0:1)$. Blowing up the singularity at the origin, we see that
$\pi^{-1}(0:0:1)$ contains two rational points in the split nodal case, no
rational point in the nonsplit nodal case, and a unique rational point in the
cuspidal case. In particular, $N_2(5,3) \le 8$.
\end{proof}

\subsection{The canonical embedding of genus-5 curves}
\label{sec:genus5_gen}

Throughout this section, $k$ denotes a perfect field with algebraic closure
$\bar k$. Given a variety $Y$, we write $\bar Y$ for the base extension $Y
\times_{\Spec k} \Spec \bar k$. Our goal for this section is to prove the
following result:

\begin{theorem}
  \label{thm:genus5_nontrig}
  Suppose $C_{/k}$ has genus~5 and is not hyperelliptic. Identify $C$
  with its image in $\PP^4$ under a canonical embedding. The $k$-vector space of
  global sections of $\OO(2)$ (i.e., quadratic forms) that vanish on $C$ is
  3-dimensional. Write $\cQ$ for the corresponding 2-dimensional space of
  quadrics in $\PP^4$.
  \begin{itemize}
  \item If $C$ is trigonal, then the intersection of all quadrics in $\cQ$ is an
    irreducible ruled surface that is $k$-isomorphic to the blowing up of $\PP^2$
    at a rational point.
  \item If $C$ is not trigonal, then the intersection of the quadrics in $\cQ$
      is $C$. 
  \end{itemize}
\end{theorem}

This result is well known when $k$ is algebraically closed and can be assembled
from the work of Max Noether and Enriques/Babbage/Petri in characteristic zero
\cite[III.3]{Geom_of_Alg_Curves} and Saint-Donat in positive characteristic
\cite{Saint-Donat_Petri}. Since the space of quadrics $\cQ$ is defined over $k$,
we obtain the following interesting consequence:

\begin{corollary}
  \label{cor:perfect_trigonal}
  If $C$ is a genus-5 non-hyperelliptic curve over a perfect field $k$, then $C$
  is trigonal if and only if it is geometrically trigonal. 
\end{corollary}

When $k = \FF_q$, the corollary follows easily from the fact that a $g^1_3$ on a
genus-5 curve is unique and that every Galois invariant divisor class over $\bar
\FF_q$ admits an $\FF_q$-rational divisor. (See, e.g.,
\cite[Rem.~2.4]{Castryck_Tuitman} or \cite[Lem.~6.5.3]{Gille_Szamuely}.) To the
best of our knowledge, we are the first to extend this to arbitrary perfect
fields.

We expend the remainder of this section in proving
Theorem~\ref{thm:genus5_nontrig}.

Let $\cI$ be the homogeneous ideal sheaf of $C$ in $\PP^4$, and write $\iota \colon
C \to \PP^4$ for the canonical closed immersion. Twist the exact
sequence
\[
0 \to \cI \to \OO_{\PP^4} \to \iota_*\OO_C \to 0
\]
by $\OO(2)$ and consider the first part of the long exact sequence on sheaf
cohomology:
\begin{equation}
  \label{eq:les}
  0 \to H^0\left(\cI(2)\right) \to H^0\left(\OO_{\PP^4}(2)\right) \to H^0\left(\OO_C(2)\right).
\end{equation}
The final homomorphism is surjective after base extending to $\bar k$ by Max
Noether's theorem in characteristic zero or Saint-Donat's theorem for positive
characteristic \cite[p.157]{Saint-Donat_Petri}. Since all of these vector spaces are
defined over $k$, it follows that \eqref{eq:les} is already surjective before
passing to $\bar k$.

Counting quadratic forms in 5 variables shows that
\[
\dim H^0\left(\OO_{\PP^4}(2)\right) = \binom{2+4}{2} = 15,
\]
while the Riemann-Roch formula implies that
\[
\dim H^0\left(\OO_C(2)\right) = \dim H^0(2K) = 2\deg(K) + 1 - 5 = 12.
\]
Thus, by \eqref{eq:les}, we see that $H^0\left(\cI(2)\right)$ has
dimension~3. This is precisely the subspace of global sections of $\OO(2)$ that
vanish on $C$, so the first part of the theorem is proved.

Write $\bar \cQ = \cQ \otimes_k \bar k$ for the space of quadrics in
$\PP^4_{\bar k}$ that contain $\bar C$. The results of
Enriques, as formulated by Saint-Donat in
\cite[(4.13)]{Saint-Donat_Petri}, say the following:
\begin{itemize}
  \item If $\bar C$ is trigonal, then the intersection of the quadrics in $\bar
    \cQ$ is isomorphic to the ruled surface $\bF_1 = \PP\left(\OO_{\PP^1} \oplus
    \OO_{\PP^1}(-1)\right)$. Moreover, the linear series cut out by the ruling
    on $\bF_1$ is a $g^1_3$.
  \item If $\bar C$ is not trigonal, then $\bar C$ is the intersection of the
    quadrics in $\bar \cQ$.
\end{itemize}

To complete the proof of Theorem~\ref{thm:genus5_nontrig}, we must show that
this result descends to $C$. Suppose first that $C$ is trigonal. Then $\bar C$
is also trigonal. Let $Y$ be the intersection of the quadrics in $\cQ$. The
above result shows that $\bar Y \cong \bF_1$. The ruled surface $\bF_1$ is
isomorphic to the blowing up of $\PP^2_k$ at a rational point
\cite[Example~V.2.11.5]{Hartshorne_Bible}, and hence has no nontrivial twist
by Lemma~\ref{lem:twist} below. That is, $Y$ is isomorphic to $\bF_1$ over $k$,
and a $g^1_3$ is given by restring the ruling on $\bF_1$ to $C$.

Conversely, suppose that $C$ is not trigonal. Let $Y$ be the intersection of the
quadrics in $\cQ$. If $\bar C$ is trigonal, then the argument in the preceding
paragraph applies to show that $Y$ is $k$-isomorphic to $\bF_1$. But then the
ruling on $\bF_1$ gives a $g^1_3$ on $C$, a contradiction. So $\bar C$ is not
trigonal, and we find that $\bar C$ is the intersection of the quadrics in $\bar
\cQ$. But the latter space has a $k$-basis, so $C$ is the intersection of the
quadrics in $\cQ$.

\begin{lemma}
  \label{lem:twist}
  Let $k$ be a perfect field. The surface $S$ given by blowing up $\PP^2_k$ at a
  rational point has no nontrivial twist. 
\end{lemma}

\begin{proof}
  Twists are in bijection with $H^1(k,\Aut(\bar S)) := H^1(\Gal(\bar k / k),
  \Aut(\bar S))$. Since $S$ has a unique curve with self-intersection $-1$, this
  curve must be stabilized by every automorphism of $S$. Blowing down the
  $(-1)$-curve yields a homomorphism $\psi : \Aut(\bar S) \to \Aut(\PP^2_{\bar
    k} ,p)$, where $p$ is the rational point of $\PP^2$ that we blew up to
  obtain $S$. In fact, $\psi$ is an isomorphism because every automorphism of
  $\PP^2_{\bar k}$ that fixes $p$ maps lines through $p$ to lines through $p$.

  Without loss, we may suppose  that $p = (1:0:0)$. The subgroup of
  $\Aut(\PP^2_{\bar k})  = \PGL_3(\bar k)$ that fixes $p$ is
  \[
    G = \left\{ \smallmat{1 & v \\ 0 & A} : v \in \bar k \oplus \bar k, A \in
    \GL_2(\bar k)\right\}.
    \]
    Evidently $G$ fits into a (split) short exact sequence 
  \[
  0 \to \bar k^2 \to G \to \GL_2(\bar k) \to 0,
  \]
  and the associated long exact sequence on Galois cohomology contains the exact
  sequence of pointed sets
  \[
  H^1(k, \bar k^2) \to H^1(k, G) \to H^1(k, \GL_2(\bar k)).
  \]
  The first term is trivial because $\bar k$ has no cohomology, and the last
  term is trivial by Hilbert 90. Hence, $H^1(k,G)$ is trivial, and $S$ has no
  nontrivial twist. 
\end{proof}


\subsection{Gonality 4}
\label{sec:5-4}

Every non-trigonal curve of genus 5 is cut out by three linearly independent
quadratic forms. It turns out that one can detect the presence of a $g^1_4$
based on the equivalence class of these forms. This is well known in the
algebraically closed setting \cite[p.207--8]{Geom_of_Alg_Curves}; we give a
proof that is agnostic to the field.

\begin{lemma}
  \label{lem:notIorII}
  Suppose that $C \subset \PP^4_{\FF_q} = \Proj \FF_q[v,w,x,y,z]$ is a
  non-trigonal canonically embedded genus-5 curve. Then $C$ has gonality 4 if and
  only if it lies on a quadric hypersurface isomorphic to $vw + x^2 = 0$ or $vw
  + xy = 0$.
\end{lemma}

\begin{proof}
  The proof is quite similar to the one used for Lemma~\ref{lem:gen4}.  Suppose
  $C$ lies on a quadric hypersurface isomorphic to $S = V(vw + x^2)$; without
  loss, we may assume $C \subset S$. This hypersurface is a cone over the
  singular quadric cone $S_0 \subset \PP^3 = \Proj \FF_q [v,w,x,y]$ given by the
  same equation. Since $S_0$ contains a 1-parameter family of lines through its
  singularity, we find that $S$ contains a 1-parameter family of 2-planes
  through the cone point $(0 : 0 : 0 : 0 : 1)$. Write $\{P_\lambda\}$ for this
  family of planes. Since $C$ is not trigonal, it is the intersection of three
  quadric surfaces; say $C = V(Q_1,Q_2,Q_3)$, where $Q_1 = vw + x^2$. It follows
  that
  \[
  P_\lambda \cap C = P_\lambda \cap V(Q_1) \cap V(Q_2) \cap V(Q_3)  
  = P_\lambda \cap V(Q_2) \cap V(Q_3),
  \]
  and the latter intersection is visibly a zero-cycle of degree~4. That is, the
  family of divisors $\{P_\lambda \cap C\}$ is a $g^1_4$ on $C$.

  The same argument applies if $C$ lies on the quadric hypersurface $xy + vw =
  0$. Such a hypersurface is isomorphic to a cone over a quadric surface in
  $\PP^3$ given by the same equation --- which is isomorphic to $\PP^1 \times
  \PP^1$ --- and this surface has a 1-parameter family of lines on it.

  For the converse, let $D$ be an effective divisor on $C$ such that $|D|$
  contains a $g^1_4$. We must consider two cases. Suppose first that $D \sim K -
  D$. Let $f \in L(D)$ be a nonconstant rational function. Then $1,f,f^2 \in
  L(2D)$. The common poles of these three functions have different orders, so
  they are linearly independent. Since $2D \sim K$ and $\dim L(K) = 5$, there
  are rational functions $h,j$ such that $1,f,f^2,h,j$ are a basis for
  $L(2D)$. Define a rational map
  \begin{eqnarray*}
    C &\dashrightarrow & \PP^4 = \Proj \FF_q[v,w,x,y,z] \\
    P &\mapsto & (-1, f^2(P), f(P), h(P), j(P)).
  \end{eqnarray*}
  The image lies on $vw + x^2 = 0$, and the fact that $C$ is smooth shows that
  this morphism extends over $\supp(D)$. Evidently, this is a canonical
  embedding of $C$.

  Finally, suppose that $D \not\sim K - D$. Then $|K-D|$ is also a $g^1_4$, and
  $\dim |D| = \dim |K-D|$ by Riemann-Roch. Let $D' \sim K - D$ be an effective
  divisor, and let $f,h$ be nonconstant rational functions in $L(D)$ and
  $L(D')$, respectively. It follows that $1, f, h, fh$ are in $L(D + D')$, and
  there is a function $j \in L(D+D')$ not in the span of these four.  Define a
  rational map
  \begin{eqnarray*}
    C  &\dashrightarrow & \PP^4 = \Proj \FF_q[v,w,x,y,z]\\
    P &\mapsto & (f(P), h(P), -1, f(P)h(P), j(P)).
  \end{eqnarray*}
  The image lies on $vw + xy = 0$, and it extends to a morphism on $C$. To show
  that the morphism gives a canonical embedding of $C$, it remains to prove that
  $1, f, h$, and $fh$ are linearly independent. The argument is identical to the
  one in the proof of Lemma~\ref{lem:gen4}. 
  
  Consider the morphism $\psi : C \to \PP^1 \times \PP^1$ induced by $|D|$ and
  $|D'|$; in coordinates, it is $P \mapsto (1,f(P)) \times (1,h(P))$.  If
  $1,f,h,fh$ are linearly dependent, then the image $C_0$ of $\psi$ is a
  $(1,1)$-curve on $\PP^1 \times \PP^1$. Such a curve has arithmetic genus~0 by
  the adjunction formula, and hence geometric genus 0. The component projections
  $\PP^1 \times \PP^1 \to \PP^1$ induce a birational morphism $C_0 \to \PP^1$,
  and they give two compositions
  \[
   C \to C_0 \to \PP^1,
   \]
   induced by $|D|$ and $|D'|$, respectively. But this implies $D \sim D' \sim K
   - D$, which is a contradiction.
\end{proof}

\begin{theorem}
   \label{thm:5-4}
  $N_2(5,4) = 9$
\end{theorem}

\begin{proof}
  Consider the subvariety of $\PP^4 = \Proj \FF_2[v,w,x,y,z]$ cut out by the equations
  \begin{eqnarray*}
    vw + xy &=& 0 \\
    vx + z(v + w + z) &=& 0 \\
    (x+y)^2 + y(v+w) &=& 0.
  \end{eqnarray*}
  A Gr\"obner basis calculation (in Sage, say) shows that these equations
  determine a smooth curve of genus~5, and one verifies directly that it has 9
  rational points. Hence, $N_2(5,4) \geq 9$.  For the upper bound, Serre showed
  that $N_2(5) \leq 9$ \cite{Serre_points_on_curves_1983}, which completes the
  proof.
\end{proof}


\subsection{Gonality at least~5}
\label{sec:genus5_calc}

In this section, we describe a computer calculation that allowed us to inspect
every curve of genus~5 over $\FF_2$ of gonality at least~5. This turned out to
be necessary for two reasons. First, we were unable to produce an upper bound on
the number of rational points on a curve with gonality~5 by purely theoretical
means. Second, we have no satisfactory explanation for the non-existence of a
curve of gonality~6; it is simply a phenomenon we observed in our data.

Let $C$ be a genus-5 curve of gonality at least~5, which we identify with its
image under a canonical embedding in $\PP^4 = \Proj \FF_2[v,w,x,y,z]$.
Theorem~\ref{thm:genus5_nontrig} shows that $C$ is the vanishing locus of a
3-dimensional space $W$ of quadratic forms. If $Q \in W$ is a nonzero form, then
it must be geometrically irreducible as $C$ does not lie on a hyperplane. The
classification of quadratic forms (Theorem~\ref{thm:classification}) shows that,
after an appropriate linear change of variable, any geometrically irreducible
quadratic form is equivalent to one of the following:
\begin{itemize}
\item[I.] $vw + x^2$ (singular line and 15 rational points)
\item[II.] $vw + xy$ (isolated singularity and 19 rational points)
\item[III.] $vw + x^2 + xy + y^2$ (isolated singularity and 11 rational points)
\item[IV.] $vw + xy + z^2$ (smooth and 15 rational points)
\end{itemize}
The parenthetic statements describe the associated quadric hypersurface.  We
will say that a form $Q$ has type I if it is equivalent to $vw + x^2$, and
similarly for types II, III, and IV. Lemma~\ref{lem:notIorII} shows that every
nonzero quadratic form in $W$ is of type~III or~IV.

To organize our search, we now argue that, up to linear change of variable, $W$
admits a special kind of basis $\{Q_1, Q_2, Q_3\}$. This involves two cases,
depending on whether $W$ contains a form of type III.

\medskip

\noindent \textbf{Case $W$ contains a form of type III.} Every nonzero form in
$W$ must be of type~III or~IV. Set $Q_1 = vw + x^2 + xy + y^2$. The orthogonal
group $O(Q_1)$ acts on the set of quadratic forms of type~III or~IV; choose a
set $A(Q_1)$ of orbit representatives for this action. Let us discard $Q_1$ from
the set $A(Q_1)$, as well as any $Q$ such that the linear span of $Q_1$ and $Q$
contains a nonzero form that is not of type~III or~IV.  Take $B(Q_1)$ to be the
set of all forms of type~III or~IV. We may take a basis for $W$ of the form $\{Q_1,
Q_2, Q_3\}$ with $Q_2 \in A(Q_1)$ and $Q_3 \in B(Q_1)$.

\medskip

\noindent \textbf{Case $W$ contains no form of type III.} Every nonzero form in
$W$ must be of type~IV. Set $Q_1 = vw + xy + z^2$. The orthogonal group $O(Q_1)$
acts on the set of quadratic forms of type~IV; choose a set $A(Q_1)$ of orbit
representatives for this action. Let us discard $Q_1$ from the set $A(Q_1)$, as
well as any $Q$ such that the linear span of $Q_1$ and $Q$ contains a nonzero
form that is not of type~IV.  Take $B(Q_1)$ to be the set of all forms of
type~IV. Evidently, we may take a basis for $W$ of the form $\{Q_1, Q_2, Q_3\}$
with $Q_2 \in A(Q_1)$ and $Q_3 \in B(Q_1)$.

\medskip

From an algorithmic standpoint, it is worth noting two things:
\begin{itemize}
  \item The type of a quadratic form can be determined by calculating the
    dimension of its singular locus and by counting its rational points. The
    former amounts to computing the rank of a $5 \times 5$ matrix over $\FF_2$,
    while the latter can be accomplished with a search over the 31 points of
    $\PP^4(\FF_2)$. This should be done once for all quadratic forms in
    $\FF_2[v,w,x,y,z]$ and stored; there are $2^{15} - 1= 32,767$ quadratic forms.
  \item We need to compute the orthogonal groups $O(vw + xy + z^2)$ and $O(vw +
    x^2 + xy + y^2)$. This can be accomplished with the techniques in
    \S\ref{sec:orthogonal}.
\end{itemize}

\begin{algorithm}[ht]
\caption{--- Compute a list of genus~5 curves over $\FF_2$ containing  all those
  of gonality at least~5, up to isomorphism}
  \begin{algorithmic}[1]
    \STATE initialize an empty list $\mathbf{curves}$.
\FOR { $Q_1 \in \{vw + x^2 + xy + y^2, vw + xy + z^2\}$ }
  \STATE compute the sets of quadratic forms $A(Q_1)$ and $B(Q_1)$.
  \FOR { $(Q_2,Q_3) \in A(Q_1) \times B(Q_1)$ }
  
  \IF {every nonzero member of the linear span of $\{Q_1,Q_2,Q_3\}$ has type at
    least that of $Q_1$, and the variety $V(Q_1,Q_2,Q_3)$ is irreducible and
    smooth of dimension~1} 
    \STATE append $(Q_1,Q_2,Q_3)$ to $\mathbf{curves}$.    
  \ENDIF
  \ENDFOR
  \ENDFOR
  \STATE \textbf{return curves}
\end{algorithmic}
  \label{alg:gonality5}
\end{algorithm}

We implemented Algorithm~\ref{alg:gonality5} in Sage. For each of the standard
forms $Q_1$, we computed the orthogonal group and the sets $A(Q_1)$ and $B(Q_1)$
and saved them to disk for later use; this required a small number of minutes of
compute time.  The loop over pairs $(Q_2,Q_3) \in A(Q_1) \times B(Q_1)$ involves
a number of commutative algebra computations that are handled by Singular, and
constitute the bulk of the runtime of the algorithm. Table~\ref{table:search_time}
summarizes the outcome of this computation. 

\begin{table}[ht]
  \begin{tabular}{c|c|c|c|c|c|c}
    $Q_1$ & $\#O(Q_1)$ & $\#A(Q_1)$ & $\#B(Q_1)$ & Curves & Wall Time \\
    \hline \hline
    $vw + x^2 + xy + y^2$ & 1,920 & 17 & 19,096 & 30,296 & 371min \\
    \hline
    $vw + xy + z^2$  & 720 & 10 & 13,888 & 8,296 & 190min \\
    \hline
  \end{tabular}
  \caption{Counts and timing in our search for all canonically embedded genus-5
    curves over $\FF_2$ of gonality at least 5, up to linear isomorphism. Every
    isomorphism class is represented by at least one curve that we found, though
    we make no claim of uniqueness of representation.}
  \label{table:search_time}
\end{table}

Each of the curves in the output of Algorithm~\ref{alg:gonality5} is presented
as the intersection of three quadratic forms. This makes it possible to count
their rational points by a straightforward search over points of
$\PP^4(\FF_2)$. The results of this procedure are presented in
Table~\ref{table:numpoints}.

\begin{table}[ht]
  \begin{tabular}{c||c|c|c|c|c}
    $Q_1$ \ $\backslash$ \ $\#C(\FF_2)$ &  0 & 1 & 2 & 3 & $\geq 4$ \\
    \hline \hline
    $vw + x^2 + xy + y^2$ &  11,864 & 13,184 & 5,248  & 0 & 0  \\ 
    \hline
    $vw + xy + z^2$  & 0  & 0 & 0 & 8,296 & 0 \\
    \hline
  \end{tabular}
  \caption{Number of curves found with Algorithm~\ref{alg:gonality5} with a given $Q_1$ and number of rational points. }
  \label{table:numpoints}
\end{table}

\begin{theorem}
  \label{thm:5-5}
  $N_2(5,5) = 3$.
\end{theorem}

\begin{proof}
  Consider the genus-5 curve $C \subset \PP^4_{\FF_2} = \Proj \FF_2[v,w,x,y,z]$
  described by the following equations:
  \begin{eqnarray*}
    vw + xy + z^2 &=& 0 \\
      vx + y^2 + vz + wz &=& 0 \\
      x^2 + wy + xy + vz + xz &=& 0.
  \end{eqnarray*}
  This is one of the curves discovered by Algorithm~\ref{alg:gonality5}, so it
  has gonality at least~5.  One verifies by a direct search that it has three
  rational points. In particular, it has gonality~5 by the third part of
  Proposition~\ref{prop:gengon}. Hence, $N_2(5,5) \ge 3$.  Looking at the data
  in Table~\ref{table:numpoints}, we see that no curve has more than 3
  rational points, so $N_2(5,5) \leq 3$.
\end{proof}

\begin{theorem}
  \label{thm:5-6}
  Every curve of genus 5 over $\FF_2$ has gonality less than or equal to~5. In
  particular, $N_2(5,6) = -\infty$.
\end{theorem}

\begin{proof}
  Algorithm~\ref{alg:gonality5} yielded 11,864 pointless curves of gonality at
  least~5. On each such curve $C$, we located a cubic
  point. Corollary~\ref{cor:gonality5} shows that each of these curves has
  gonality~5.
\end{proof}


\appendix

\section{Miscellaneous values of $N_2(g,\gamma)$ for $6 \le g \le 10$}
\label{sec:misc}

We determine $N_2(g,\gamma)$ in a number of additional cases by techniques that
do not fit the main narrative of the above article. Table~\ref{tab:adhoc}
summarizes our findings. Our main tool is the ``gonality-point inequality'' from
\eqref{eq:gamma_bound}:
\begin{equation*}
C \text{ has gonality $\gamma$} \quad \Longrightarrow \quad \#C(\FF_2) \leq 3 \gamma.
\end{equation*}

\begin{table}[th]
  \begin{tabular}{c|c|c}
  $g$ & $\gamma$ & $N_2(g,\gamma)$ \\
    \hline \hline
    6 & 2 & 6 \\
     & 3 & 9 \\
     & 4 & 10 \\
     & 5--6 & $< 10$ \\
     & 7 & $\le 0$ \\    
    \hline
    7 & 2 & 6 \\
     & 3 & 9 \\
     & 4 & 10\\
     & 5--7 & $\le 10$\\
     & 8 & $\le 0$\\    
    \hline
    8 & 2 & 6 \\
      & 3 & 9 \\
      & 4 & 11\\
      & 5--9 & $\le 11$\\
      & 9 & $\le 0$\\    
    \hline
    9 & 2 & 6 \\
     & 3 & 9 \\
     & 4 & 12\\
     & 5--9 & $\le 12$ \\
     & 10 & $\le 0$ \\    
    \hline
    10 & 2 & 6 \\
     & 3 & 9 \\
    & 4--10 & $\le 13$\\
     & 11 & $\le 0$\\        
  \end{tabular}
  \caption{Maximum number of rational points on binary curves with fixed genus
    and gonality.}
  \label{tab:adhoc}
\end{table}

All of the entries in Table~\ref{tab:adhoc} with $\gamma = 2$ follow from
Theorem~\ref{thm:hyperelliptic}. Each of the entries with $\gamma = 3$ must
satisfy $N_2(g,3) \leq 9$ by the gonality-point inequality.  The following
examples achieve this bound. We discovered them via a naive search for
polynomials $f \in \FF_2[x,y]$ with $y$-degree 3 and small $x$-degree. In each
case, the equation $\{f = 0\}$ gives a (typically singular) affine plane model
for the curve $C$, and the $x$-coordinate function provides a morphism to
$\PP^1$ of degree~3. We used Magma \cite{magma} to compute the genus and count
the rational points on the smooth model; each example has 9 rational
points. Note that if $C$ were hyperelliptic, then $\#C(\FF_2) \leq 6$ by the
gonality-point inequality. In all cases, we conclude that $N_2(g,3) = 9$.

\begin{example}[genus 6, gonality 3, 9 rational points]
  \label{ex:6-3}
  \[
    C_{/\FF_2} : (x^3 + x)y^3 + (x^4 + x + 1)y^2 + (x^4 + x^3 + 1)y + x^3 + x^2 = 0
  \]    
\end{example}

\begin{example}[Genus 7, gonality 3, 9 rational points]
  \label{ex:7-3}  
  \[
    C_{/\FF_2} : y^3 + (x^5 + x^2 + x + 1)y^2 + (x^7 + x^6 + x^2)y + x^7 + x^6    = 0
    \]
\end{example}

\begin{example}[Genus 8, gonality 3, 9 rational points]
  \label{ex:8-3}
  \[
    C_{/\FF_2} : y^3 + (x^6 + 1)y^2 + (x^7 + x^6 + x^2)y + x^7 + x^6
    = 0
    \]
\end{example}

\begin{example}[Genus 9, gonality 3, 9 rational points]
  \label{ex:9-3}  
  \[
    C_{/\FF_2} : y^3 + (x^6 + 1)y^2 + (x^7 + x^5 + x)y + x^7 + x^6
    = 0
    \]
\end{example}

\begin{example}[Genus 10, gonality 3, 9 rational points]
  \label{ex:10-3}  
  \[
    C_{/\FF_2} : y^3 + (x^7 + x^6 + x^5 + 1)y^2 + (x^7 + x^5 + x)y + x^6 + x^5
    = 0
    \]
\end{example}

The remaining known entries of Table~\ref{tab:adhoc} have gonality $\gamma =
4$. We look through the entries with small genus over $\FF_2$ on
\url{manypoints.org} to find additional curves that meet our needs.

\begin{example}[Genus 6, gonality 4, 10 rational points]
  Steve Fischer gave the following example in 2014:
  \begin{eqnarray*}
    X_{/\FF_2} &:& (x^3 + x^2 )y^3 + (x^2+ x +1)y + 1 = 0 \\
    C_{/\FF_2} &:& z^2 + x^2z + (x^3 + x^2)y = 0.
  \end{eqnarray*}
  Here $X$ is a curve of genus~2 with 6 rational points, and $C$ is a double
  cover of $X$ with 10 rational points. Serre showed that $N_2(6) = 10$. Since
  $C$ is a double cover of a genus-2 curve, it must have gonality at most~4. If
  its gonality $\gamma$ were at most~3, then the gonality-point inequality would show
  $\#C(\FF_2) \leq 9$. Hence its gonality is exactly~4, and we have $N_2(6,4) =
  10$.
\end{example}

\begin{example}[Genus 7, gonality 4, 10 rational points]
  Another example by Steve Fischer:
  \begin{eqnarray*}
    X_{/\FF_2} &:& ( x^3 + x )y^3 + (x^3 + x^2 + x)y + 1 = 0 \\
    C_{/\FF_2} &:& z^2 + z + x^3 + x = 0. 
  \end{eqnarray*}
  Just as in the case of genus~6, we find $N_2(7,4) \ge 10$. Serre showed that
  $N_2(7) \le 10$; therefore, $N_2(7,4) = 10$. 
\end{example}

\begin{example}[Genus 8, gonality 4, 11 rational points]
  Isabel Pirsic found the following example in 2012:
  \begin{align*}
    C_{/\FF_2} :  (x^8+x^4+1)y^4 &+ (x^9+x^6+x^5+x+1)y^2 \\
    & +
  (x^9+x^8+x^6+x^5+x^4+x)y  + (x^8+x^7+x^4+x^3) = 0
  \end{align*}
  The $x$-coordinate function gives a degree-4 morphism to $\PP^1$, and $C$
  cannot have gonality smaller than~4 by the gonality-point inequality. Hence, $N_2(8,4)
  \geq 11$.  Serre showed that $N_2(8) = 11$, so $N_2(8,4) = 11$.
\end{example}

\begin{example}[Genus 9, gonality 4, 12 rational points]
  Another example by Isabel Pirsic:
  \begin{align*}
    C_{/\FF_2} \ \colon \ & (x^{12} + x^{10} + x^8 + x^6 + x^4 + x^2 + 1)y^4 \\
    &\hspace*{1cm} + (x^{14} + x^{13} + x^{11} + x^9 + x^7 + x^5 + x^3 + x + 1)y^2 \\
  &\hspace*{1cm}+ (x^{14} + x^{13} + x^{12} + x^{11} + x^{10} + x^9 + x^8 + x^7 + x^6 +x^5 +
  x^4 + x^3 + x^2 + x)y \\
  &\hspace*{1cm}+ x^{11} + x^9 + x^7 + x^5 = 0
  \end{align*}
Just as in the genus~8 case, we find that $N_2(9,4) \ge 12$. Serre showed that
$N_2(9) \le 12$, so we find that $N_2(9,4) = 12$.
\end{example}

One would like to extend the same kind of reasoning to genus~10. However,
$N_2(10) = 13$, so any example that achieves this bound necessarily has gonality
at least 5. Steve Fischer produced a curve with 13 rational points via iterated
double covers that has a natural morphism to $\PP^1$ of degree~6. We conclude that
$N_2(10,5) = 13$ \textit{or} $N_2(10,6) = 13$, but we are unable to determine
which of these is the truth without additional work.

Finally, we address the entry in the table with $g = 6$ and $\gamma > 4$. In
\cite{Rigato_optimal_curves}, the author proves that there are exactly two
curves of genus~6 with 10 rational points, up to isomorphism. Her arguments show
that one of the curves, say $C_1$, is a double cover of an elliptic curve, and
hence has gonality~4; it does not seem immediately obvious how to suss out the
gonality of the other. We claim that the other curve, $C_2$, has gonality 4 as
well. On \url{manypoints.org}, Steve Fischer gave the following example of a
genus-6 curve over $\FF_2$ with 10 rational points:
  \begin{eqnarray*}
    X &:& (x^3 + x^2 )y^3 + (x^2+ x +1)y + 1 = 0 \\
    C_2 &:& z^2 + x^2z + (x^3 + x^2)y = 0.
  \end{eqnarray*}
Here $X$ is a genus-2 curve with 6 rational points, and $C_2$ is a double cover
of it. That is, $C_2$ has gonality 4. One can verify, using Magma say, that
$\#C_2(\FF_{32}) = 20$. Since $\#C_1(\FF_{32}) = 25$, as one can read off of the
data for $a(X)$ in \cite[Thm.~2.4]{Rigato_optimal_curves}, we see that $C_1$ and
$C_2$ must be non-isomorphic. It follows that both of the isomorphism classes of
curves with 10 rational points have gonality~4, and hence any curve of genus 6
with gonality $5, 6$, or $7$ must have fewer than 10 points. 


\section{Quadratic Forms over Finite Fields}
\label{app:quadratic_forms}

Literature on the classification of non-degenerate quadratic forms in odd
characteristic is abundant, but it is much harder to find a self-contained
reference for general quadratic forms in all characteristics. Once one lets go
of the idea that the associated bilinear form should retain all of the
information about the quadratic form, the theory becomes quite streamlined in
all characteristics. Following unpublished notes of Bill Casselman
\cite{Casselman_quadratic_forms} -- who essentially follows the treatment in
\cite{Elman_et_al} --- we give a summary of all of the results that we need. For
additional reference, see \cite{Serre_Course_in_Arithmetic} or \cite{Arf_char2}.


\subsection{Classification of quadratic forms}

\begin{definition}
  Let $\FF_q$ be a finite field and let
  \[
  Q(x) = \sum_{1 \leq i\leq j \leq n} c_{i,j} x_i x_j
  \]
  be a quadratic form over $\FF_q$ in $n$ variables, where $x = (x_1, \ldots,
  x_n)$ and $c_{i,j} \in \FF_q$ are not all zero. Here $n$ is the
  \textbf{dimension} of $Q$. The associated bilinear form is
  \[
  \bil{Q}{x}{y} = Q(x+y) - Q(x) - Q(y).
  \]
  The \textbf{radical} of the bilinear form $\bilname{Q}$ is defined to be
  \[
  \rad = \left\{ a \in \FF_q^n \ :\  \bil{Q}{a}{\FF_q^n} = 0\right\}.
  \]
  A quadratic form is called \textbf{strictly non-degenerate} if $\rad = 0$, or
  equivalently, if $\bilname{}$ is a perfect pairing.
\end{definition}

\begin{remark} If we write $e_i$ for the $i$-th standard basis vector of $\FF_q^n$, then
\[
\bil{Q}{a}{e_i} = \frac{\partial Q}{\partial x_i}(a),
\]
Consequently, the radical is precisely the set of $\FF_q$-rational
points at which all partial derivatives of $Q$ vanish.
\end{remark}

For the remainder of this section, we fix $\alpha \in \FF_{q^2}$ such that
$\FF_{q^2} = \FF_q(\alpha)$. The \textbf{norm form},
\begin{equation}
  \label{eq:norm_form}
\N(x_1,x_2) := \N_{\FF_{q^2}/\FF_q}(x_1 + \alpha x_2) = x_1^2 + (\alpha + \alpha^q)x_1x_2 + \alpha^{q+1}x_2^2,
\end{equation}
is a strictly non-degenerate quadratic form over $\FF_q$.

We say that two quadratic forms $Q_1, Q_2$ are \textbf{equivalent} if they agree
up to a linear change of variables on their domain.

\begin{proposition}[Strictly non-degenerate forms]
  Let $Q$ be a strictly non-degenerate quadratic form over $\FF_q$ of
  dimension~$n$. 
  \begin{itemize}
    \item If $n$ is even, then $Q$ is equivalent to one of the following two
      forms:
      \begin{eqnarray*}
        && x_1x_2 + x_3x_4 + \cdots + x_{n-1}x_n \\
        && x_1x_2 + x_3x_4 + \cdots + x_{n-3}x_{n-2} + a\N(x_{n-1},x_{n-2}),
      \end{eqnarray*}
      where $a \in \FF_q^\times$. The first and second
      forms are never equivalent, and two of the latter type of form with
      final coefficients $a,a'$ are equivalent if and only if $a/a'$ is a square in
      $\FF_q^\times$. 
    \item If $n$ is odd, then $q$ must also be odd, and $Q$ is equivalent to
      \[
      x_1x_2 + x_3x_4 + \cdots + x_{n-2}x_{n-1} + ax_n^2
      \]
      for some $a \in \FF_q^\times$. Two such forms with final coefficients $a,a'$
      are equivalent if and only if $a/a'$ is a square in $\FF_q^\times$. 
  \end{itemize}
\end{proposition}

Now we deal with the general case. Fix a quadratic form $Q$ over $\FF_q$ of
dimension~$n$. Take a linear complement $U$ to $\radQ$ in $\FF_q^n$, so that
$\FF_q^n = U \oplus \radQ$. Since every element of $\FF_q^n$ is orthogonal to
the radical for $\bilname{}$, this is an orthogonal direct sum, and the
restriction of $Q$ to $U$ is strictly non-degenerate. Moreover, we see that if
$u \in U$ and $v \in \radQ$, then
\[
Q(u + v) = \bil{Q}{u}{v} + Q(u) + Q(v) = Q(u) + Q(v).
\]
So the quadratic form decomposes additively over this direct sum. The above
proposition characterizes $Q|_U$, and the next proposition characterizes the
restriction of $Q$ to the radical.

\begin{proposition}[Totally degenerate forms]
  Let $Q$ be a quadratic form over $\FF_q$ of dimension~$n$ such that $\radQ =
  \FF_q^n$.
  \begin{itemize}
  \item If $q$ is odd, then $Q = 0$.
    \item If $q$ is even, then either $Q = 0$ or $Q$ is equivalent to $x_1^2$.
  \end{itemize}  
\end{proposition}

\begin{proof}
  All partial derivatives of $Q$ vanish identically by our assumption on the
  radical. In the odd characteristic case, this shows $Q(x) = 0$. In the even
  characteristic case, $Q(x) = \sum c_i x_i^2$ for some $c_i \in \FF_q$. The
  squaring map is bijective on finite fields of even characteristic, so for each
  $i$ with $c_i \ne 0$, we may replace $x_i$ with $x_i / \sqrt{c_i}$ in order to
  assume all of the $c_i \in \{0,1\}$. Then
  \[
  Q(x) = \sum_{i=1}^n c_i x_i^2 = \left( \sum_{i=1}^n c_i x_i \right)^2.
  \]
  If all $c_i = 0$, we are finished. Otherwise, we move $c_1x_1 + \cdots + c_n
  x_n$ to $x_1$ to obtain the result.
\end{proof}

Combining the previous two results gives a complete classification:

\begin{theorem}[Classification of quadratic forms]
  \label{thm:classification}
  Let $Q$ be a quadratic form over $\FF_q$ of dimension~$n$. There is $m \leq n$
  for which $Q$ is equivalent to one of
    \begin{eqnarray}
      && x_1x_2 + x_3x_4 + \cdots + x_{m-1}x_m, \label{eq:hyperbolic} \\
      && x_1x_2 + x_3x_4 + \cdots + x_{m-3}x_{m-2} + a\N(x_{m-1},x_m), \label{eq:norm_term}\\
      && x_1x_2 + x_3x_4 + \cdots + x_{m-2}x_{m-1} + ax_m^2 \label{eq:sqr_term}
      \end{eqnarray}
      where $a \in \FF_q^\times$.
      \begin{itemize}
      \item If $Q$ is equivalent to \eqref{eq:hyperbolic} or
        \eqref{eq:norm_term}, or if $q$ is odd and $Q$ is equivalent to
        \eqref{eq:sqr_term},  then $\radQ = \{x_1 = x_2
        = \cdots = x_m = 0\}$.
      \item If $q$ is even and $Q$ is equivalent to \eqref{eq:sqr_term}, then
        $\radQ = \{x_1 = x_2 = \cdots = x_{m-1} = 0\}$.
      \end{itemize}
\end{theorem}

\begin{proof}
  The result is immediate from the two preceding propositions when $q$ is odd or
  when $q$ is even and $Q|_{\radQ} = 0$. So we are reduced to considering the
  case where $q$ is even, $\FF_q^n = U \oplus \radQ$ with $\dim U = m-1$, and
  $Q|_{\radQ} = x_m^2$. Note that $Q|_U$ is strictly non-degenerate. If $Q|_U$
  is equivalent to $x_1x_2 + \cdots + x_{m-2}x_{m-1}$, we are finished. Suppose
  instead that $Q|_U$ is equivalent to
  \[
  x_1x_2 + x_3x_4 + \cdots + x_{m-4}x_{m-3} + a\N(x_{m-2},x_{m-1})
  \]
  for some $a \in \FF_q^{\times}$. Since the squaring map is surjective, we may
  absorb $\sqrt{a}$ into $x_{m-2}$ and $x_{m-1}$ in order to assume $a =
  1$.

  To complete the proof, it suffices to show that the form $\N(x,y) + z^2$ is
  equivalent to $xy + z^2$. If $\N(x,y) = x^2 + Axy + By^2$, then the linear
  change of variables
  \[
  x \mapsto x + y, \qquad y \mapsto \frac{1}{A} y, \qquad
  z \mapsto x + \frac{\sqrt{B}}{A}y + z
  \]
  does the trick. Note that $A \ne 0$ and $\sqrt{B} \in \FF_q$ since $q$ is even. 
\end{proof}


\subsection{Orthogonal groups}
\label{sec:orthogonal}

Now we look at the group of linear transformations that preserves a given
quadratic form.

\begin{definition}
  Let $Q$ be a quadratic form over $\FF_q$ of dimension~$n$. The
  \textbf{orthogonal group} of $Q$ is defined to be
  \[
  O(Q) = \left\{ g \in \GL_n(\FF_q) \ : \ Q(g(x)) = Q(x) \right\}.
  \]
\end{definition}

An immediate consequence of the definition is that the associated bilinear form
is also preserved by elements of the orthogonal group:
\[
\bil{Q}{g(x)}{g(y)} = \bil{Q}{x}{y} \qquad \text{for all $g \in O(q)$}.
\]
Consequently, the orthogonal group preserves the radical of $\bilname{Q}$.

\begin{theorem}[Witt Extension Theorem]
Let $Q$ be a quadratic form over $\FF_q$ of dimension~$n$ with radical $\radQ$,
and let $U_1, U_2 \subset \FF_q^n$ be subspaces such that $U_1 \cap \radQ = U_2
\cap \radQ = 0$. Then any isometry $U_1 \to U_2$ may be extended to an element
of the orthogonal group of $Q$.
\end{theorem}

We will only need to use this result on 1-dimensional subspaces, or
equivalently, on points of projective space. We define two special linear
subvarieties of interest. Let $Q$ be a quadratic form over $\FF_q$ of dimension
$n+1$. Denote the set of $\FF_q$-rational points on the quadric hypersurface
$V(Q) \subset \PP^n$ by
\[
\cQ = \{a \in \PP^n(\FF_q) \ : \ Q(a) = 0\}.
\]
The projectivization of the radical will be denoted
\[
\cR = \PP(\radQ)(\FF_q) \subset \PP^n(\FF_q).
\]
The linear subvariety $\PP(\radQ)$ will be referred to as the \textbf{radical
  locus}.  Note that the orthogonal group acts on both $\cQ$ and $\cR$.

Looking at 1-dimensional subspaces of $\FF_q^{n+1}$, the Witt Extension Theorem
implies that if $a,b \in \cQ \smallsetminus \cR$, then there exists $g \in O(Q)$
such that $g(a) = g(b)$. We would like to augment this result to include a
simultaneous transitivity statement on $\cR$. There is one small wrinkle that
must be addressed before stating the result.

As the orthogonal group acts on $\cQ$ and $\cR$, it also acts on their
intersection
\[
\cS = \cQ \cap \cR \subset \PP^n(\FF_q).
\]
This is the set of $\FF_q$-rational points at which $Q$ and all of its partial
derivatives vanish; that is, $\cS$ is the set of $\FF_q$-rational singular
points of the hypersurface $V(Q)$. Our transitivity result is the best possible
result on points that takes into account these subspaces:

\begin{theorem}[Transitivity]
  \label{thm:transitive}
Let $Q$ be a quadratic form over $\FF_q$ of dimension $n+1$, and let $\cQ, \cR,
\cS$ be the $\FF_q$-rational points of $V(Q)$, of the radical locus, and of the
singular locus of $V(Q)$, respectively. Write $\cY$ for the product of all of
the nonempty sets among
\[
\left(\cQ \smallsetminus \cS\right), \qquad  \left(\cR \smallsetminus \cS\right), \qquad \cS.
\]
The orthogonal group $O(Q)$ acts transitively on $\cY$.
\end{theorem}

\begin{example}
  Suppose that $Q$ is equivalent to \eqref{eq:sqr_term} for some $m \le n+1$. If
  $q$ is odd and $m = n+1$, then $\cR = \cS = \varnothing$ and $\cY = \cQ$. At
  the other extreme, if $q$ is even and $1 < m < n+1$, then $\varnothing
  \subsetneq \cS \subsetneq \cR$, $\cQ \ne \cS$, and $\cY = (\cQ \smallsetminus
  \cS) \times (\cR \smallsetminus \cS) \times \cS$.
\end{example}

Recall that $Q$ is additive when restricted to $\radQ$, since
\[
Q(a+b) = \bil{Q}{a}{b} + Q(a) + Q(b) = Q(a) + Q(b).
\]
The subset of $\radQ$ on which $Q$ vanishes is therefore an $\FF_q$-subspace,
which we will call $S$. It follows that $\cS = \PP(S)(\FF_q)$. Before proving
the theorem, we require two lemmas: the first describes how the singular locus
sits inside the radical locus, and the second gives a block decomposition of
$O(Q)$ along $S$ and its orthogonal complement.

\begin{lemma}[Singular Locus]
  Let $S \subset \radQ$ be the subspace on which $Q$ vanishes. Then
  \[
  \codim(S,\radQ) \leq 1.
  \]
Moreover, $\codim(S,\radQ) = 1$ precisely when $q$ is even and $Q$ is equivalent
to
\[
x_1 x_2 + \cdots + x_{m-2} x_{m-1} + x_m^2
\]
for some $m \leq n+1$. 
\end{lemma}

\begin{proof}
  Decompose $\FF_q^{n+1} = U \oplus \radQ$, and let $m = \dim U$. Suppose first
  that $q$ is odd, or that $q$ is even and $m$ is even. Then the Classification
  of Quadratic Forms shows that $Q$ is equivalent to one of
\begin{eqnarray*}
&& x_1 x_2 + \cdots + x_{m-1} x_m, \\
&& x_1 x_2 + \cdots + x_{m-3} x_{m-2} + a\N(x_{m-1},x_m), \\
&& x_1 x_2 + \cdots + x_{m-2} x_{m-1} + ax_m^2,
\end{eqnarray*}
for some $a \in \FF_q^\times$. The latter case can only occur when $q$ is
odd. In these coordinates, we have
\begin{eqnarray*}
    U &=& \{x_{m+1} = \cdots = x_{n+1} = 0\} \\
\radQ &=& \{x_1 = \cdots = x_m = 0\}
\end{eqnarray*}
Evidently, $S = \radQ$. 

By the Classification of Quadratic Forms, the only remaining case is when $q$ is
even, $m$ is odd, and $Q$ is equivalent to 
\[
x_1 x_2 + \cdots + x_{m-2} x_{m-1} + x_m^2.
\]
(Since squaring is onto in characteristic~2, we can absorb the coefficient $a$ in
\eqref{eq:sqr_term}.) In these coordinates, we have
\begin{eqnarray*}
    U &=& \{x_m = \cdots = x_{n+1} = 0\} \\
\radQ &=& \{x_1 = \cdots = x_{m-1} = 0\}
\end{eqnarray*}
Here, we find $S = \{x_1 = \cdots = x_m = 0\}$, so that $S$ has
codimension~1 inside $\radQ$.
\end{proof}

\begin{lemma}[Block Decomposition of Isometries]
  Let $Q$ be a quadratic form over $\FF_q$ of dimension $n+1$. Let $S$ be the
  maximal subspace of $\radQ$ on which $Q$ vanishes. If we decompose
  $\FF_q^{n+1}$ as $U \oplus S$, then an element $g \in \GL_{n+1}(\FF_q)$ lies
  in the orthogonal group of $Q$ if and only if it admits a decomposition as
  \[
    g = \begin{pmatrix} A & 0 \\ B & C \end{pmatrix},
    \]
  where $A \in O(Q|_U)$, $B : U \to S$ is an arbitrary linear map, and $C \in
  \GL(S)$.
\end{lemma}

\begin{proof}
One implication is just a computation: if $g$ has the desired form, $u \in U$ and $s \in
S$, then
\begin{align*}
  Q\left(g(u+s)\right) &= Q\left(A(u) + B(u) + C(s)\right) \\
  &= Q\left(A(u)\right), \text{ since $Q$ is additive across
    $U \oplus S$ and kills $S$, } \\
  &= Q(u).
\end{align*}

Turning to the other direction, we take $g \in O(Q)$ and write it in block form
  \[
    g = \begin{pmatrix} A & D \\ B & C \end{pmatrix},
  \]
 where $A : U \to U$, $D : S \to U$, $B : U \to S$, and $C : S \to S$ are all
 linear.  We must show that $A \in O(Q|_U)$, that $D = 0$, and that $C$ is
 invertible.

 Using the fact that $Q$ is additive across $U \oplus S$ again, for any $u \in
 U$ and $s \in S$ we have
 \begin{align*}
   Q(u) &= Q(u + s) = Q\left(g(u+s)\right) \\
   &= Q\left(A(u) + D(s) + B(u) + C(s) \right) \\
   &= Q\left(A(u) + D(s)\right) \\
   &= \bil{Q}{A(u)}{D(s)} + Q\left(A(u)\right) + Q\left(D(s)\right).
 \end{align*}
 Setting $s = 0$ shows that  $A \in O(Q|_U)$. In particular, we obtain the
 equation
 \begin{equation}
   \label{eq:block}
 \bil{Q}{A(u)}{D(s)} = - Q\left(D(s)\right) \qquad (u \in U, s \in S).
 \end{equation}

 For the sake of a contradiction, suppose that $D \ne 0$ and fix $s_0 \in S$
 such that $D(s_0) \ne 0$. Setting $u = 0$ in \eqref{eq:block} shows that
 $Q(D(s_0)) = 0$. Since $A$ is invertible, it is onto, and we see from
 \eqref{eq:block} that $\bil{}{U}{D(s_0)} = 0$. Consequently, $D(s_0) \in U \cap
 \radQ$.

 If $\radQ = S$, then $D(s_0) \in U \cap S = 0$, a contradiction. Thus, we must
 have $\radQ \ne S$. The Singular Locus lemma shows that $q$ is even, and that
 in an appropriate choice of coordinates we have
 \[
 Q|_U = x_1 x_2 + \cdots + x_{m-2} x_{m-1} + x_m^2.
 \]
 As $D(s_0) \in U \cap \radQ$, we find that $D(s_0) = ae_m$, where $a \in
 \FF_q^\times$ and $e_m$ is the $m$-th standard basis vector. But now $Q(D(s)) =
 a^2 \ne 0$, another contradiction. We must concede that our initial assertion,
 namely $D \ne 0$, was false.

 Finally, we see that $C$ must be invertible, for otherwise it would have some
 nonzero element $s$ in its kernel, and then the the equation $g(s) = C(s) = 0$
 would contradict the invertibility of $g$.
\end{proof}

\begin{proof}[Proof of the Transitivity Theorem]
Decompose $\FF_q^{n+1} = U \oplus S$, where $S$ is the subspace of $\radQ$ on
which $Q$ vanishes. Let us, if necessary, replace the standard basis with one
that has its first $m$ vectors in $U$ and its last $n+1-m$ vectors in $S$. By
the definition of $S$, we know $Q|_S = 0$, so $Q$ depends only on $x_1, \ldots,
x_m$.

In the remainder of the proof, if $a \in \PP^n(\FF_q)$, we will write $\tilde a$
for a fixed choice of lift to $\FF_q^{n+1}$. If $\tilde a$ is a lift,
so is $t\cdot \tilde a$ for any $t \in \FF_q^\times$.

Consider first the case where $S = \radQ$. Then $\cR = \cS$. Choose $u_1, u_2
\in \cQ \smallsetminus \cR$ and $v_1, v_2 \in \cR$. We wish to produce an
isometry $g \in O(Q)$ such that $g(u_1) = u_2$ and $g(v_1) = v_2$. To that end,
we may write $\tilde u_i = \tilde a_i + \tilde b_i$ with $\tilde a_i \in U$ and
$\tilde b_i \in S$. Since $u_i \not\in \cR$, we have $\tilde a_i \ne 0$. Since
$S = \radQ$, we know that $Q|_U$ is strictly non-degenerate. By the Witt
Extension Theorem applied to $Q|_U$, there is $A \in O(Q|_U)$ such that
$A(\tilde a_1) = \alpha \tilde a_2$ for some $\alpha \in \FF_q^\times$. Choose
any $C \in \GL(S)$ such that $C(\tilde v_1) = \tilde v_2$. Finally, since
$\tilde a_1 \ne 0$, we can choose a linear map $B$ such that $B(\tilde a_1) =
\alpha \tilde b_2 - C(\tilde b_1)$. Define $g$ to be a block matrix with entries
$A,0,B,C$ as in the lemma. Then we have
  \[
    g(\tilde u_1) = \begin{pmatrix} A & 0 \\ B & C \end{pmatrix} \begin{pmatrix}
      \tilde a_1 \\ \tilde b_1 \end{pmatrix}
    = A(\tilde a_1) + B(\tilde a_1) + C(\tilde b_1) = \alpha \tilde a_2 + \alpha
    \tilde b_2 = \alpha \tilde u_2,
    \]
and similarly, $g(\tilde v_1) = \tilde v_2$. This completes the proof of the
theorem when $S = \radQ$.

In the remaining case, we have $q$ even and $S \subset \radQ$ has
codimension~1. In particular, after an appropriate choice of coordinates, our
quadratic form is
\[
Q = x_1x_2 + \cdots + x_{m-2}x_{m-1} + x_m^2, 
\]
where $m \le n + 1$. In these coordinates, we have
\begin{eqnarray*}
U &=& \{x_{m+1} = \cdots = x_{n+1} = 0\} \\
S &=& \{x_1 = \cdots = x_m = 0\} \\
\radQ &=& \{x_1 = \cdots = x_{m-1} = 0\}.
\end{eqnarray*}

Choose $u_1, u_2 \in \cQ \smallsetminus \cR$, $s_1, s_2 \in \cS$, and $v_1, v_2
\in \cR \smallsetminus \cS$. Choose lifts of all of these. Write $\tilde u_i =
\tilde a_i + \tilde b_i$ with $\tilde a_i \in U$ and $\tilde b_i \in S$. Since
$u_i \not \in \cR$, it follows that $\tilde a_i \not\in \radQ$. By Witt's
Extension Theorem, there is $A \in O(Q|_{U})$ such that $A(\tilde a_1) = \alpha
\tilde a_2$ for some $\alpha \in \FF_q^\times$. Choose any $C \in \GL(S)$ such
that $C(\tilde s_1) = \tilde s_2$.

Now let's look at $\tilde v_1$ and $\tilde v_2$. As these lie in $\radQ$, but
not in $S$, it follows that the $m$-th entry of each of them is nonzero. For
convenience, we may rescale their lifts so that they both have $m$-th entry
equal to~1. Now $\tilde v_1 = e_m + w_1$ and $\tilde v_2 = e_m + w_2$, where
$e_m$ is the $m$-th standard basis vector for $\FF_q^{n+1}$ and $w_1, w_2 \in
S$. (We do not decorate the $w_i$ with tildes because one or both of them could
be identically zero, and hence may not be lifts of elements of $\cS$.) Now
observe that $U \cap \radQ$ is the radical of $Q|_U$. In particular, $A$
preserves this subspace, so we must have $A(e_m) = \beta e_m$ for some $\beta
\in \FF_q^\times$. From the previous paragraph, we saw that $\tilde a_1 \not\in
\radQ$, and hence $\tilde a_1$ and $e_m$ are linearly independent.  Thus, we are
able to choose a linear map $B : U \to S$ such that
\begin{eqnarray*}
  B(\tilde a_1) &=& \alpha \tilde b_2 - C(\tilde b_1)  \\
  B(e_m) &=& \beta w_2 - C(w_1).
\end{eqnarray*}
 If we define $g$ to be a block matrix with entries $A, 0, B, C$ as in the
 lemma, then it follows as before that $g(\tilde u_1) = \alpha \tilde u_2$,
 $g(\tilde v_1) = \beta \tilde v_2$, and $g(\tilde s_1) = \tilde s_2$, which
 completes the proof.
\end{proof}

We close with a discussion of how to compute the orthogonal group of a quadratic
form $Q$ of dimension~$n+1$. The first sensible thing to do is to loop over the
$(n+1)\times(n+1)$ matrices $g$ with coefficients in $\FF_q$ and test whether
$g$ is invertible and whether $Q(g(x)) = Q(x)$. Since there are $q^{(n+1)^2}$
such matrices, this is only practical for small $n$ and small $q$.

\begin{example}
  \label{ex:naive_search}
  Consider the quadratic form over $\FF_2$ of dimension~5 given by
  \[
  Q(v,w,x,y,z) = vw + x^2.
  \]
  There are $2^{25} \approx 10^{7.5}$ matrices to search through.  Running Sage
  on a 2.6 GHz Intel Core i5, the naive search took approximately 54 minutes to compute $O(Q)$.
\end{example}

Typically, a more efficient method for computing $O(Q)$ can be given by using
our Transitivity Theorem.

\noindent
\textbf{Step 1.} Compute $\cQ$. To accomplish this, we loop over the the
$\frac{q^{n+1} - 1}{q-1}$ points $x \in \PP^n(\FF_q)$ and keep $x$ if $Q(x) = 0$.

\noindent
\textbf{Step 2.} Compute $\cR$ and $\cS$. The set $\cR$ can be computed by
linear algebra: it is the vanishing locus of the $n+1$ partial derivatives of
$Q$. Then we set $\cS = \cQ \cap \cR$.

\noindent
\textbf{Step 3.} Compute $O(Q)$. Let $\cY$ be as in the Transitivity Theorem;
let $i \in \{1,2,3\}$ be the number of sets in the product defining $\cY$. Fix
$p_0 \in \cY$. For $g \in \GL_{n+1}(\FF_q)$, write $g.p_0$ for the image of
$p_0$ in $(\PP^n)^i$. For any $p \in \cY$, the equation $g.p_0 = p$ provides
$i(n+1)$ linear equations in $(n+1)^2 + i$ unknowns: $(n+1)^2$ from the entries
of $g$ and $i$ coming from the scaling factors inherent in working with $i$
points of $\PP^n$. Linear algebra gives a space of solutions of dimension
$(n+1)^2 -in$. Looping over $(n+1)\times (n+1)$ matrices $g$ that satisfy these
conditions and recording those that are invertible and satisfy $Q(g(x)) = Q(x)$,
we obtain $O(Q)$.

The computations in Steps~1 and~2 are negligible compared to the search in
Step~3.  The search space in this latter step has size
\[
\#\cY \cdot q^{(n+1)^2 -in},
\]
which --- depending on $Q$ --- may or may not compare favorably to the naive
bound given by $\#\GL_{n+1}(\FF_q)$. By computing $\cY$, we can determine in
advance whether this is a better strategy before doing the search.

\begin{example}
  Returning to Example~\ref{ex:naive_search}, we find that
  $\#(\cQ \smallsetminus \cS) = 12$, $\#(\cR \smallsetminus \cS) = 4$, and
  $\#\cS = 3$. Thus, $\#\cY = 144$, and the search space involved in Step~3 of
  the above algorithm is $144 \cdot 2^{25 - 3\cdot 4} = 1179648 \approx
  10^{6.07}$. This search space is $10^{7.5 - 6.07} \approx 28.4$ times smaller
  than the naive search search. Running Sage on the same computer as in
  Example~\ref{ex:naive_search}, this improved algorithm took around 4~minutes
  to compute $O(Q)$.
\end{example}


\bibliographystyle{plain}
\bibliography{binary}
\end{document}